%% file: circuit_and_co-circuit_meeting_infinitely.tex
\documentclass[11pt]{article}

\input{minimalJ}

\usepackage{verbatim}

\usepackage{subfig}

\title{\scshape Matroids with an infinite circuit-cocircuit intersection}
\author{Nathan Bowler \and Johannes Carmesin}

\begin{document}
 
\maketitle

 \begin{abstract}
 We construct some matroids that have a circuit and a cocircuit with infinite intersection.

This answers a question of Bruhn, Diestel, Kriesell, Pendavingh and Wollan.
It further shows that the axiom system for matroids proposed by Dress in 1986 does not axiomatize all infinite matroids.

We show that one of the matroids we define is a thin sums matroid whose dual is not a thin sums matroid,
answering a natural open problem in the theory of thin sums matroids.

\end{abstract}

\section{Introduction}

In \cite{matroid_axioms}, Bruhn, Diestel, Kriesell, Pendavingh and Wollan introduced axioms for infinite matroids 
in terms of independent sets, bases, circuits, closure and (relative) rank. 
These axioms allow for duality of infinite matroids as known from finite matroid theory, which settled an old problem of Rado.
Unlike the infinite matroids known previously, such matroids can have infinite circuits or infinite cocircuits.
Many infinite matroids are \emph{finitary}, that is, every circuit is finite, or \emph{cofinitary}, that is, every cocircuit is finite, 
but nontrivial matroids with both infinite circuits and infinite cocircuits have been known for some time \cite{{matroid_axioms},{Higgs69_2}}.

However in all the known examples, all intersections of circuit with cocircuit are finite.
Moreover, this finiteness seems to be a natural requirement in many theorems \cite{{THINSUMS},{BC:rep_matroids}}.
This phenomenon prompted the authors of \cite{matroid_axioms} to ask the following.

\begin{que}[\cite{matroid_axioms}]\label{open}
Is the intersection of a circuit with a cocircuit in an infinite matroid always finite?
\end{que}

Dress \cite{dress} even thought that the very aim to have infinite matroids with duality, as in Rado's problem, 
would make it necessary that circuit-cocircuit intersection were finite. 
He therefore proposed axioms for infinite matroids which had the finiteness of circuit-cocircuit intersections built into the definition 
of a matroid, in order to facilitate duality.


And indeed, it was later shown by Wagowski \cite{wagowski} that the axioms proposed by Dress capture
all infinite matroids as axiomatised in \cite{matroid_axioms} if and only if \autoref{open} has a positive answer.
We prove that the assertion of \autoref{open} is false and consequently that the axiom system for matroids proposed by Dress does not capture all matroids.


\begin{thm}\label{main}
There exists a matroid $M$ that has a circuit $C$ and a cocircuit $D$ such that $|C\cap D|=\infty$.
\end{thm}

To construct such matroids $M$, we use some recent result from an investigation of matroid union \cite{union2}. 

We call a matroid \emph{tame} if the intersection of any circuit with any cocircuit is finite, and otherwise \emph{wild}.
We hope that the wild matroids we construct here may be sufficiently badly behaved
to serve as generic counterexamples also for other open problems.
To illustrate this potential, we shall show that we do obtain a counterexample
to a natural open question about thin sums matroids, a generalisation of
representable matroids.

If we have a family of vectors in a vector space, we get a matroid structure on that family whose independent sets are given by 
the linearly independent subsets of the family. Matroids arising in this way are called {\em representable} matroids. Although 
many interesting finite matroids (eg. all graphic matroids) are representable, it is clear that any representable matroid is finitary and so many 
interesting examples of infinite matroids are not of this type. However, since the construction of many of these examples, including the algebraic cycle matroids of infinite graphs, is suggestively similar to that of representable matroids, the notion of {\em thin sums matroids} was introduced in \cite{RD:HB:graphmatroids}: it is a generalisation of representability which captures these infinite examples.

Since thin sums matroids need not be finitary, and the duals of many thin sums matroids are again thin sums matroids, it is natural to ask whether the class of thin sums matroids itself is closed under duality. 
It is shown in \cite{THINSUMS} that the class of tame thin sums matroids is closed under duality, 
so that any counterexample must be wild. We show below that one of the wild matroids we have constructed does give a counterexample.

\begin{thm}
There exists a thin sums matroid whose dual is not a thin sums matroid.
\end{thm}

The paper is organised as follows.
In Section 2, we recall some basic matroid theory.
After this, in Section 3, we give the first example of a wild matroid.
In Section 4, we give a second example, which is obtained by taking the union of a matroid with itself.
In Section 5, we show that the class of thin sums matroids is not
closed under duality by construct a suitable wild thin sums matroid whose dual is not a thin sums matroid.

\section{Preliminaries}\label{prelim}

Throughout, notation and terminology for graphs are that of~\cite{DiestelBook10}, for matroids  that of~\cite{Oxley,matroid_axioms}.
A set system $\Ical$ is the set of independent sets of a matroid if it satisfies the following \emph{independence axioms}~\cite{matroid_axioms}.
\begin{itemize}
	\item[(I1)] $\emptyset\in \Ical$.
	\item[(I2)] $\Ical$ is closed under taking subsets.
	\item[(I3)] Whenever $I,I'\in \Ical$ with $I'$ maximal and $I$ not maximal, there exists an $x\in I'\setminus I$ such that $I+x\in \Ical$.
	\item[(IM)] Whenever $I\subseteq X\subseteq E$ and $I\in\Ical$, the set $\{I'\in\Ical\mid I\subseteq I'\subseteq X\}$ has a maximal element.
\end{itemize}

$M$ always denotes a matroid and $E(M)$, $\Ical(M)$, $\Bcal(M)$, $\Ccal(M)$ and $\Scal(M)$ denote its ground set and its sets of 
independent sets, bases, circuits and spanning sets, respectively.
A matroid is called \emph{finitary} if every circuit is finite.

In our constructions, we will make use of algebraic cycle matroid $M_A(G)$ of a graph $G$.
The circuits of $M_A(G)$ are the edge sets of finite cycles of $G$
and the edge sets of double rays\footnote{A \emph{double ray} is a two sided infinite path}.
If $G$ is locally finite, then $M_A(G)$ is \emph{cofinitary}, that is, its dual is finitary \cite{union2}.
If $G$ is not locally finite, then this is no longer true \cite{matroid_axioms}.
Higgs \cite{Higgs69_2} characterized those graphs $G$ that have an algebraic cycle matroid, that is, whose finite circuits and double rays from the circuits of a matroid:
$G$ has an algebraic cycle matroid if and only if $G$ does not contain a
subdivision of the Bean-graph, see \autoref{bean}.

   \begin{figure} [htpb]   
\begin{center}
   	  \includegraphics[width=10cm]{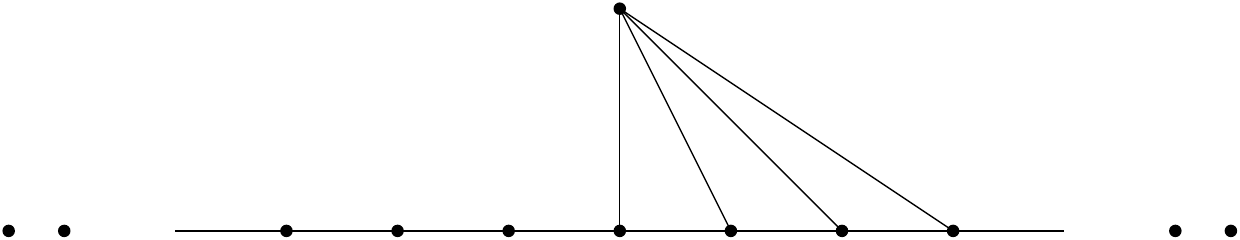}
   	  \caption{The Bean-graph}
   	  \label{bean}
\end{center}
   \end{figure}

\section{First construction: the matroid $M^+$}\label{+}

In this example, we will need the following construction from \cite{union2}:

\begin{dfn}\label{M+}
Let $M$ be a matroid, in which $\emptyset$ isn't a base. Then the matroid $M^-$, on the same groundset, is that whose bases are those obtained by removing a point from a base of $M$. That is, $\Bcal(M^-) = \{B - e | B \in \Bcal(M), e \in B\}$. Dually, if $M$ is a matroid whose ground set $E$ isn't a base, we define $M^+$ by $\Bcal(M^+) = \{B + e | B \in \Bcal(M), e \in E \setminus B\}$.
\end{dfn}

Thus $(M^+)^* = (M^*)^-$.


We shall show that the matroids constructed in this way are very often wild.

Since $M^-$ is obtained from $M$ by making the bases of $M$ into dependent sets, we may expect that $\Ccal(M^-) = \Ccal(M) \cup \Bcal(M)$: 
that is, the set of circuits of $M^-$ contains exactly the circuits and the bases of $M$. This is essentially true, but there is one complication: 
an $M$-circuit might include an $M$-base, which would prevent it being an $M^-$-circuit.  Let $O$ be a circuit of $M^-$. If $O$ is $M$-independent, 
it is clear that $O$ must be an $M$-base. Conversely, any $M$-base is a circuit of $M^-$. If $O$ is $M$-dependent, then since all proper subsets of 
$O$ are $M^-$-independent and so $M$-independent, $O$ must be an $M$-circuit. Conversely, an $M$-circuit not including an $M$-base is an $M^-$-circuit. 

 On the other hand, none of the circuits of $M$ is a circuit of $M^+$: for any circuit $O$ of $M$, pick any $e \in O$ and extend $O - e$ 
to a base $B$ of $M$. Then $O \subseteq B + e$, so $O \in \Ical(M^+)$. In fact, a circuit of $M^+$ is a set minimal with the property that 
at least two elements must be removed before it becomes $M$-independent. 
To see this note that the independent sets of $M^+$ are those sets from which an $M$-independent set can be obtained by removing at most one element.

Now we are in a position to construct a wild matroid: 
let $M$ be the algebraic cycle matroid of the graph in \autoref{fig:second_wild}. Then the dashed edges form a circuit in $M^+$, and the bold edges form a circuit in $(M^+)^* = (M^*)^-$ (they form a base in $M^*$ since their complement forms a base in $M$). The intersection, consisting of the dotted bold edges, is evidently infinite.

   \begin{figure} [htpb]   
\begin{center}
   	  \includegraphics[width=10cm]{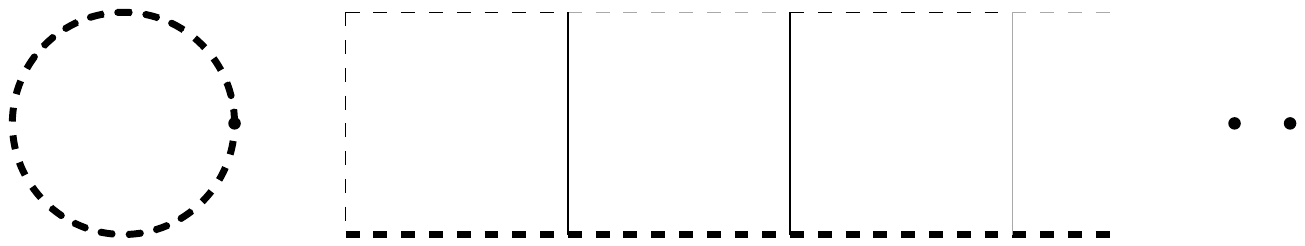}
   	  \caption{A circuit and a cocircuit with infinite intersection}
   	  \label{fig:second_wild}
\end{center}
   \end{figure}

For the remainder of this section, we will generalize this example to construct a large class of wild matroids.
To do so, we first have a closer look at the circuits of $M^+$.
It is clear that if $M$ is the finite cycle matroid of a graph $G$, 
then we get as circuits of $M^+$ any subgraphs which are subdivisions of those
in \autoref{fig:fin_cyc}.

  \begin{figure} [htpb]   
\begin{center}
   	  \includegraphics[width=10cm]{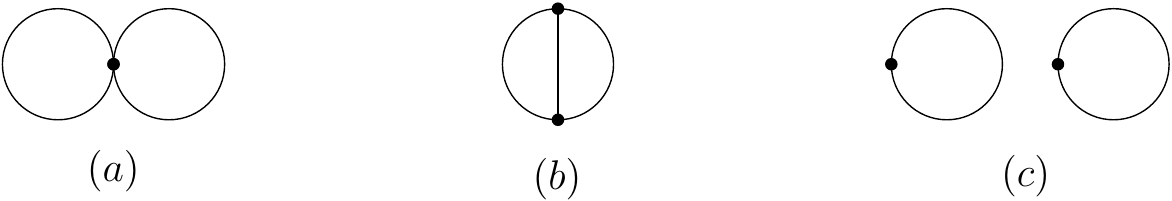}
   	  \caption{Shapes of circuits in $M^+$, with $M$ a finite cycle matroid}
   	  \label{fig:fin_cyc}
\end{center}
   \end{figure}

More generally, we can make precise a sense in which every circuit of $M^+$ is obtained by sticking together two circuits.

\begin{lem}\label{ind+}
Let $O$ be a circuit of $M$, and $I \subseteq E(M) \setminus O$. Then $O \cup I$
is $M^+$-independent iff $I$ is $M/O$-independent.
\end{lem}
\begin{proof}
If: Extend $I$ to a base $B$ of $M/O$. Pick any $e \in O$. Then $B' = B \cup O - e$ is a basis of $M$ and $O \cup I \subseteq B' + e$.

Only if: Pick $B$ a base of $M$ and $e \in E \setminus B$ such that $O \cup I \subseteq B \cup e$. Since $O$ is dependent, we must have $e \in O$, and so $I \subseteq B \setminus O$. Finally, $B \setminus O$ is a base of $M/O$, since $B \cap O = O - e$ is a base of $O$.
\end{proof}

\begin{lem}\label{circ+}
Let $O_1$ be a circuit of $M$, and $O_2$ a circuit of $M/O_1$. Then $O_1 \cup O_2$ is a circuit of $M^+$. Every circuit of $M^+$ arises in this way.
\end{lem}
\begin{proof}
$O_1 \cup O_2$ is $M^+$-dependent by Lemma \ref{ind+}. Next, we shall show that any set $O_1 \cup O_2 - e$ obtained by removing a single element from $O_1 \cup O_2$ is $M^+$-independent, and so that $O_1 \cup O_2$ is a {\em minimal} dependent set (a circuit) in $M^+$. The case $e \in O_2$ is immediate by Lemma \ref{ind+}. If $e \in O_1$, then we pick any $e' \in O_2$. Now extend $O_2 -e'$ to a base $B$ of $M / O_1$. Then $B' = B \cup O_1 - e$ is a base of $M$ and $O_1 \cup O_2 - e \subseteq B' + e'$.

Finally, we need to show that any circuit $O$ of $M^+$ arises in this way. $O$ must be $M$-dependent, and so we can find a circuit $O_1 \subseteq O$ of $M$. Let $O_2 = O \setminus O_1$: $O_2$ is a circuit of $M/O_1$ by Lemma \ref{ind+}.
\end{proof}

\begin{cor}\label{2circs}
Any union of two distinct circuits of $M$ is dependent in $M^+$.
\end{cor}

It follows from Lemma \ref{circ+} that the subgraphs of the types illustrated in \autoref{fig:fin_cyc} give all of the circuits of $M^+$ for $M$ a finite cycle matroid. Similarly, subdivisions of the graphs in \autoref{fig:fin_cyc} and \autoref{fig:alg_cyc} give circuits in the algebraic cycle matroid of a graph.

  \begin{figure} [htpb]   
\begin{center}
   	  \includegraphics[width=10cm]{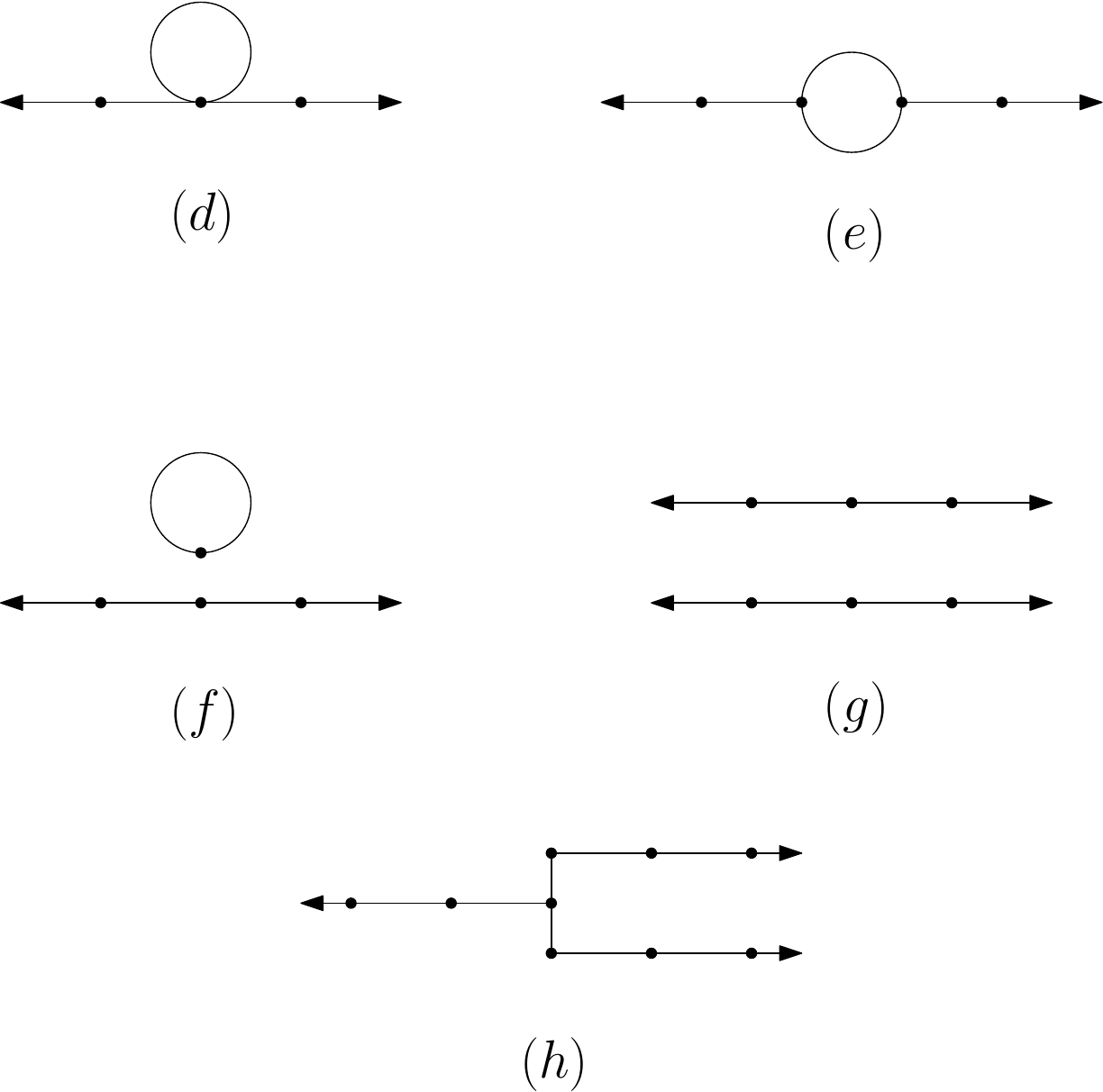}
   	  \caption{Shapes of circuits in $M^+$, with $M$ an algebraic cycle matroid}
   	  \label{fig:alg_cyc}
\end{center}
   \end{figure}

Now that we have a good understanding of the circuits of matroids constructed this way, we can find many matroids $M$ such that $M^+$ is wild. 

\begin{thm}
Let $M$ be a matroid such that
\begin{enumerate}
 \item $M$ contains at least two circuits;
\item $M$ has a base $B$ and a circuit $O$ such that $O\setminus B$ is infinite.
\end{enumerate}
Then $M^+$ is wild.
\end{thm}

\begin{proof}
Let $O'$ be any circuit other than $O$.
As $O'$ is dependent in $M/O$, there is an $M/O$-circuit $O''$ contained in
$O'$. By \autoref{circ+}, $O\cup O''$ is an $M^+$-circuit.

Since $E\backslash B$ is an $M^*$-base, it is a circuit of $(M^*)^-=(M^+)^*$.
Now $(O\cup O'')\cap (E\backslash B)$ includes $O\setminus B$ and so it is
infinite.

\end{proof}

\section{Second construction: matroid union}

The \emph{union of two matroids $M_1 = (E_1,\Ical_1)$ and $M_2=(E_2,\Ical_2)$} is the pair $(E_1\cup E_2, \Ical_1\vee \Ical_2)$, where
\[
 \Ical_1\vee \Ical_2:=  \{I_1\cup I_2\mid I_1 \in \Ical_1,\; I_2 \in \Ical_2\}
\]

The \emph{finitarization $M^{fin}$ of a matroid $M$} is the matroid whose circuits are precisely the finite circuits of $M$.
In \cite{union2} it is shown that $M^{fin}$ is always a matroid.
Note that every base of $M^{fin}$ contains some base of $M$ and conversely every base of $M$ is contained in some base of $M^{fin}$.
A matroid $M$ is called \emph{nearly finitary} if for every base of $M$, it suffices to add finitely many elements to that base to obtain some base of $M^{fin}$.
It is easy to show that $M$ is nearly finitary if and only if  
for every base of $M^{fin}$ it suffices to delete finitely many elements from that base to obtain some base of $M$.

The main tool for this example is the following theorem.

\begin{thm}[\cite{union2}]\label{nf}
The union of two nearly finitary matroids is a matroid, and in fact nearly finitary.
\end{thm}

Note that there are two matroids whose union is not a matroid \cite{union1}.

One can also define $M^+$ using matroid union: $M^+=M\vee U_{1,E(M)}$.
Here $U_{1,E(M)}$ is the matroid with groundset $E(M)$, whose bases are the $1$-element subsets of $E(M)$.
In this section, we will obtain a wild matroid as union of some non-wild matroid $M$ with itself.

Let us start constructing $M$.
We obtain the graph $H$ from the infinite one-sided ladder $L$ by doubling every edge, see \autoref{fig:cit_co-cir}.

   \begin{figure} [htpb]   
\begin{center}
   	  \includegraphics[width=10cm]{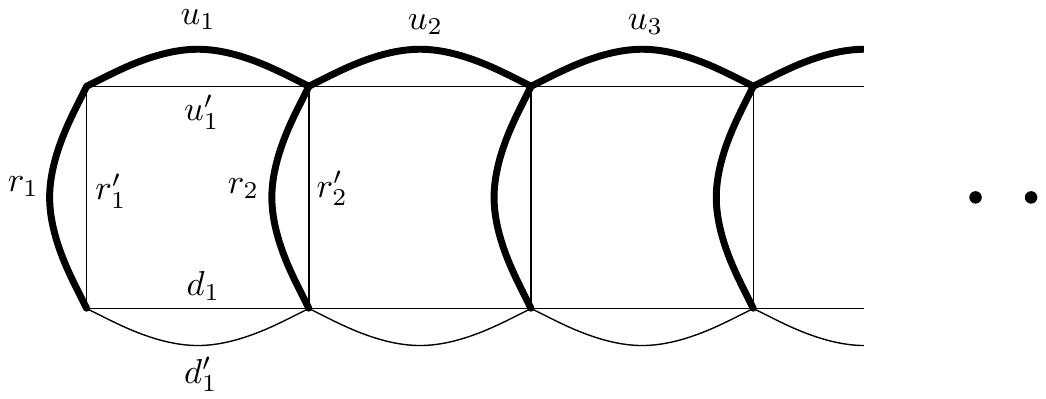}
   	  \caption{The graph $H$}
   	  \label{fig:cit_co-cir}
\end{center}
   \end{figure}

As in the figure, we fix the following notation for the edges of $H$:
In $L$, call the edges on the upper side of the ladder $u_1,u_2,\ldots$, the edges on the lower side $d_1,d_2,\ldots$ and the rungs $r_1,r_2,\ldots$.
For every edge $e$ of $L$, call its clone $e'$.

Let $M_A(H)$ be the algebraic cycle matroid of $H$. Note that $M_A(H)$ is a matroid by the results mentioned in the Preliminaries.
Now we define $M$ as the union of $M_A(H)$ with itself.
To show that $M$ is a matroid, by \autoref{nf} it suffices to show the following.

\begin{lem}\label{H_is_nf}
$M_A(H)$ is nearly finitary. 
\end{lem}

\begin{proof}
First note that the finitarization of $M_A(H)$ is the finite cycle matroid $M_F(H)$, whose circuits are the finite cycles of $H$.
To see that $M_A(H)$ is nearly finitary, it suffices to show that each base $B$ of $M_F(H)$ contains at most one double ray.
It is easy to see that a double ray $R$ of $H$ contains precisely one rung $r_i$ or $r_i'$. 
From this rung onwards, $R$ contains precisely one of $u_j$ or $u_j'$ for $j\geq i$ and one of  $d_j$ or $d_j'$ for $j\geq i$.
Let $R$ and $S$ be two distinct double rays with unique rung edges $e_R$ and $e_S$.
Wlog assume that the index of $e_R$ is less or equal than the index of $e_S$.
Then already $R+e_S$ contains a finite circuit, which consists of $e_R$, $e_S$ and all edges of $R$ with smaller index than that of $e_S$.
So each base $B$ of $M_F(H)$ contains at most one double ray, proving the assumption.
\end{proof}

Having proved that $M\vee M$ is a matroid, we next prove that it is wild.

\begin{thm}\label{union_wild}
The matroid $M\vee M$ is wild.
\end{thm}

To prove this, we will construct a circuit $C$ and a cocircuit $D$ with infinite intersection.
Let us start with $C$, which we define as the set of all horizontal edges in \autoref{fig:cit_co-cir} together with the rung $r_1$.

\begin{lem}\label{C}
$C:=\{u_i, u_i', d_i,d_i'|i=1,2,\ldots\}+r_1 $ is a circuit of $M$.
\end{lem}
\begin{proof}
First, we show that $C$ is dependent. To this end, it suffices to show that $C-r_1= \{u_i, u_i', d_i,d_i'|i=1,2,\ldots\}$ is a basis of $M$.
As $I_1=\{u_i, d_i|i=1,2,\ldots\}$ and $I_2=\{u_i',d_i'|i=1,2,\ldots\}$ are both independent in $M_A(H)$, their union $C-r_1$ is independent in $M$.
All other representations $C-r_1=I_1\cup I_2$ with $I_1,I_2\in \Ical(M_A(H))$ are the upper one up to exchanging parallel edges since 
from $u_i$ and $u_i'$ precisely one is $I_1$ and the other is in $I_2$. Similarly, the same is true for $d_i$ and $d_i'$.
So $C-r_1$ is a base and $C$ is dependent, as desired.

It remains to show that $C-e$ is independent for every $e\in C$.
The case $e=r_1$, was already consider above.
By symmetry, we may else assume that $e=u_i$. Then $C-u_i=I_1\cup I_2$ where
$I_1=\{u_i, d_i|i=1,2,\ldots\}-u_i+r_1$ and $I_2=\{u_i',d_i'|i=1,2,\ldots\}$ and $I_1$ and $I_2$ are both independent in $M_A(H)$, proving the assumption. 
\end{proof}

Next we turn to $D$, drawn bold in \autoref{fig:cit_co-cir}.

\begin{lem}\label{D}
$D:=\{u_i, r_i|i=1,2,\ldots\}$ is a cocircuit of $M$.
\end{lem}

\begin{proof}
 To this end, we show that $E\setminus D$ is a hyperplane, that is, $E\setminus D$ 
is non-spanning and $E\setminus D$ together with any edge is spanning in $M$.
To see that $E\setminus D$ is non-spanning, we properly cover it by the following two bases $B_1$ and $B_2$ of $M_A(H)$, see \autoref{fig:cover1}.
Formally, $B_1:=\{d_i|i=1,2,\ldots \} \cup \{r_i'|i \text{ odd} \}\cup \{u_i'|i \text{ odd} \}$, $B_2:=\{d_i'|i=1,2,\ldots\} \cup \{r_i'|i \text{ even}\}\cup \{u_i'|i \text{ even}\}+r_1$.

   \begin{figure} [htpb]   
\begin{center}
   	  \includegraphics[width=10cm]{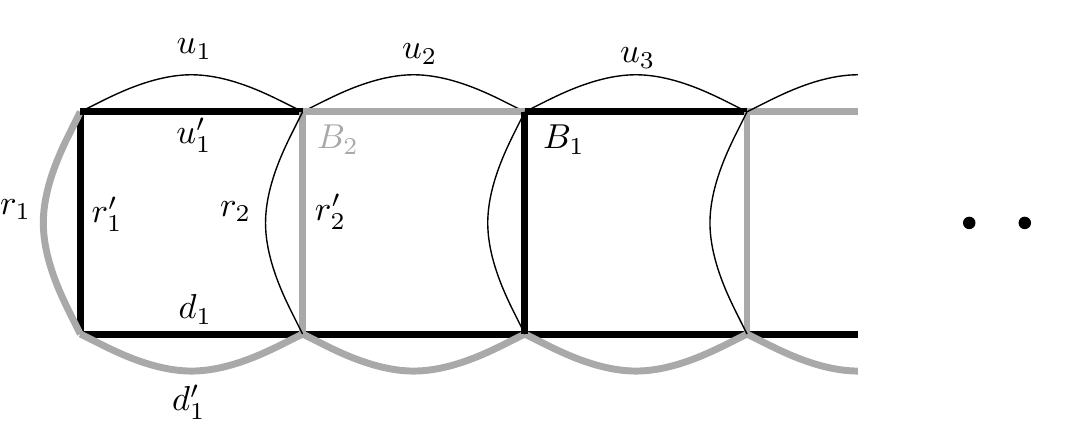}
   	  \caption{The two bases $B_1$ and $B_2$ properly cover $E\setminus D$}
   	  \label{fig:cover1}
\end{center}
   \end{figure}

To see that $E\setminus D$ together with any edge is spanning in $M$, we even show that  $E\setminus D$ together with any edge is a base of $M$.
This is done in two steps: first we show that $E\setminus D$ together with any edge $e$ is independent in $M$ and then that 
$E\setminus D$ together with any two edges is dependent in $M$.
Concerning the first assertion, we distinguish between the cases $e=u_n$ for some $n$ and $e=r_n$ for some $n$.
In both cases we assume that $n$ is odd. If $n$ is even, then the argument is similar.
In both cases we will cover $E\setminus D+e$ with two bases of $M_A(H)$, which arise from a slight modification of $B_1$ and $B_2$, see Figures \ref{fig:cover2} and \ref{fig:cover3}.
{
   \begin{figure} [htpb]   
\begin{center}
   	  \includegraphics[width=10cm]{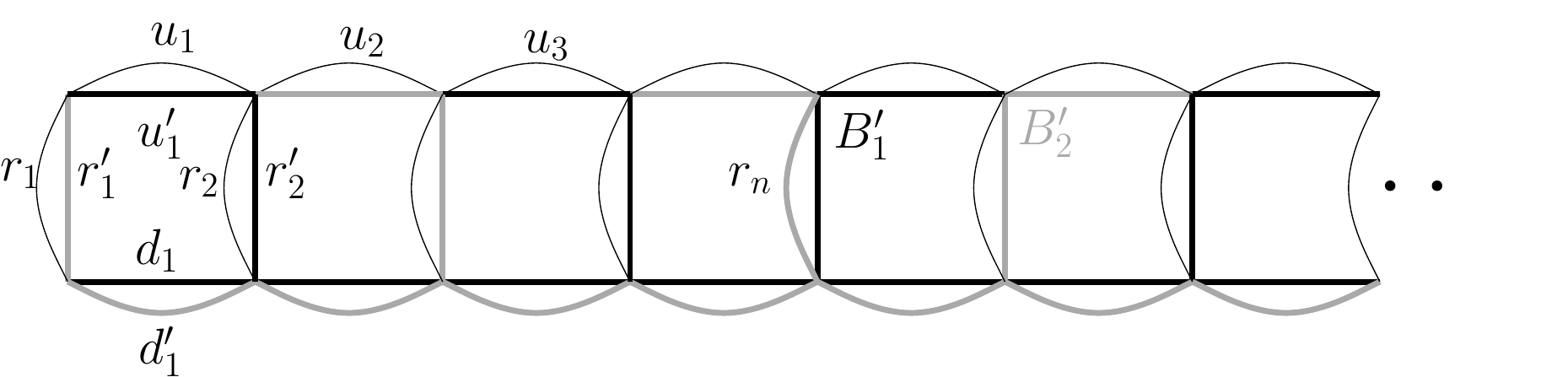}
   	  \caption{The two bases $B_1'$ and $B_2'$ cover $E\setminus D+r_n$}
   	  \label{fig:cover2}
\end{center}
   \end{figure}
   \begin{figure} [htpb]   
\begin{center}
   	  \includegraphics[width=10cm]{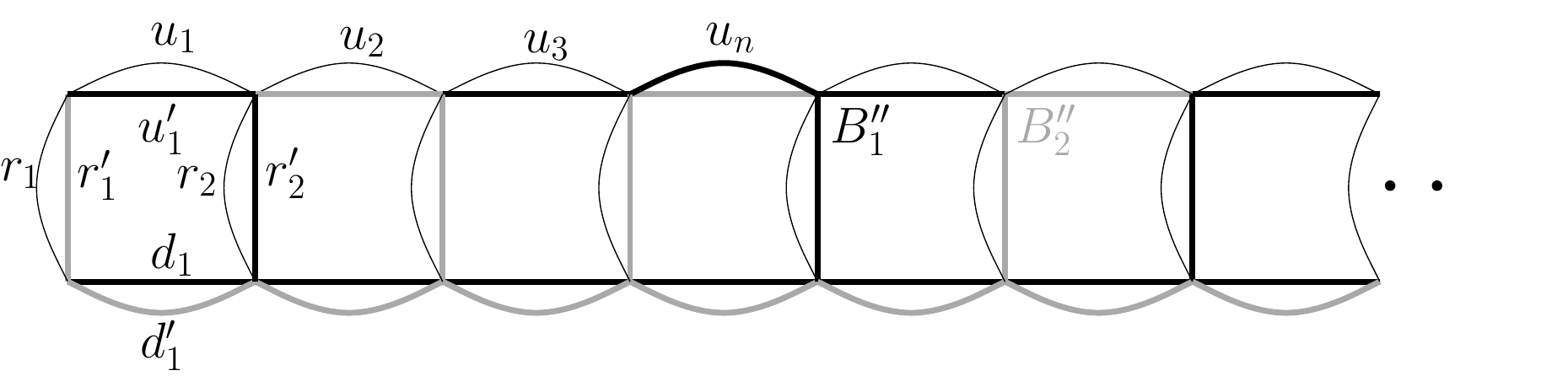}
   	  \caption{The two bases $B_1''$ and $B_2''$ cover $E\setminus D+u_n$}
   	  \label{fig:cover3}
\end{center}
   \end{figure}
}

In the first case the bases are
{
\[
 B_1':=B_1\setminus \{r_i'| i< n \text{ and odd}\} \cup \{r_i'| i< n \text{ and even}\}, 
\]
\[
B_2':=B_2 \cup \{r_i'| i< n \text{ and odd}\} \setminus \{r_i'| i< n \text{ and even}\}-r_1+r_n 
\]
}
In the second case, the bases arise from $B_1'$ and $B_2'$ as follows:
\[
 B_1'':= B_1'+u_n-r_{n-1}', \  \ B_1'':= B_1'+r_{n-1}'-r_n
\]

Having shown that $E\setminus D+e$ is independent for every $e\in D$, it remains to show for any two $e_1,e_2\in D$ that 
$E\setminus D+e_1+e_2$ cannot be covered by two bases of $M_A(H)$.
In fact we prove the slightly stronger fact that $E\setminus D+e_1+e_2$ cannot be covered by two bases of $M_F(H)$, that is by two spanning trees $T_1$ and $T_2$ of $H$.

Let $H_n$ be the subgraph of $H$ consisting of those $2n$ vertices that have the least distance to $r_1$.
Choose $n$ large enough so that $e_1,e_2\in H_n$. 
An induction argument shows that $E\setminus D$ has $4n-3$ edges in $H_n$ 
since $E\setminus D$ has $1$ edge in $H_1$ and $E\setminus D$ has $4$ edges is $H_n\setminus H_{n-1}$.
On the other hand, $T_1 \cup T_2$ can have at most $2(2n-1)$ edges in $H_n$ since $H_n$ has $2n$ vertices.
This shows that $T_1$ and $T_2$ cannot cover $E\setminus D+e_1+e_2$ because they cannot cover $H_n\setminus D+e_1+e_2$.
So for any $e\in D$ the set $E\setminus D+e$ is a base of $M$, proving the assumption.
\end{proof}

As $|C\cap D|=\infty$, this completes the proof of \autoref{union_wild}.

In the previous section, we were able to generalise our example and give a
necessary condition under which $M^+$ is wild.
Here, we do not see a way to do this, because the description of $C$ and $D$ made heavy use of the structure of $M$.
It would be nice to have a large class of matroids $M$, as in the previous
section, such that $M\vee M$ is wild.

\begin{oque}
For which matroids $M$ is $M\vee M$ wild?
\end{oque}

\section{A thin sums matroid whose dual isn't a thin sums matroid}
The constructions introduced so far give us examples of matroids which are wild, and so badly behaved. We therefore believe they will be a fruitful source of counterexamples in matroid theory. In this section, we shall illustrate this by giving a counterexample for a very natural question. 

First we recall the notion of a thin sums matroid. 

\begin{dfn}
Let $A$ be a set, and $k$ a field. Let $f = (f_e | e \in E)$ be a family of functions from $A$ to $k$, and let $\lambda = (\lambda_e | e \in E)$ be a family of elements of $k$. We say that $\lambda$ is a {\em thin dependence} of $f$ iff for each $a \in A$ we have $$\sum_{e \in E} \lambda_e f_e(a) = 0 \, ,$$
where the equation is taken implicitly to include the claim that the sum on the left is well defined, that is, that there are only finitely many $e \in E$ with $\lambda_e f_e(a) \neq 0$. 

We say that a subset $I$ of $E$ is {\em thinly independent} for $f$ iff the only
thin dependence of 
$f$ which is 0 everywhere outside $I$ is $(0 | e \in E)$. The {\em thin sums
system} $M_f$ of $f$ is the set of such thinly independent sets. This isn't
always the set of independent sets of a matroid \cite{matroid_axioms}, but when
it is we call it the {\em thin sums matroid} of $f$.
\end{dfn}

This definition is deceptively similar to the definition of the representable matroid corresponding to $f$ considered as a family of vectors in the $k$-vector space $k^A$. The difference is in the more liberal definition of dependence: it is possible for $\lambda$ to be a thin dependence even if there are infinitely many $e \in E$ with $\lambda_e \neq 0$, provided that for each $a \in A$ there are only finitely many $e \in E$ such that {\em both} $\lambda_e \neq 0$ and $f_e(a) \neq 0$. 

Indeed, the notion of thin sums matroid was introduced as a generalisation of the notion of representable matroid: every representable matroid is finitary, but this restriction does not apply to thin sums matroids. Thus, although it is clear that the class of representable matroids isn't closed under duality, the question of whether the class of thin sums matroids is closed under duality remained open. It is shown in \cite{THINSUMS} that the class of tame thin sums matroids is closed under duality, so that any counterexample must be wild. We show below that one of the wild matroids we have constructed does give a counterexample.

There are many natural examples of thin sums matroids: for example, the algebraic cycle matroid of any graph not including a subdivision of the Bean graph is a thin sums matroid, as follows:

\begin{dfn}\label{rep_A}
Let $G$ be a graph with vertex set $V$ and edge set $E$, and $k$ a field. We can pick a direction for each edge $e$, calling one of its ends its {\em source} $s(e)$ and the other its {\em target} $t(e)$. Then the family $f^G = (f^G_e | e \in E)$ of functions from $V$ to $k$ is given by $f_e = \chi_{t(e)} - \chi_{s(e)}$, where for any vertex $v$ the function $\chi_v$ takes the value 1 at $v$ and 0 elsewhere.
\end{dfn}

\begin{thm}
Let $G$ be a graph not including any subdivision of the Bean graph. Then $M_{f^G}$ is the algebraic cycle matroid of $G$.
\end{thm}

This theorem, which motivated the defintion of $M_f$, is proved in \cite{THINSUMS}. 

For the rest of this section, $M$ will denote the algebraic cycle matroid for the graph $G$ in \autoref{fig:labelled_ladder}, 
in which we have assigned directions to all the edges and labelled them for future reference. 
We showed in the Section \ref{+} that $M^+$ is wild. We shall devote the rest of
this Section \ref{+} to showing that in fact it gives an example of a thin sums
matroid whose dual isn't a thin sums matroid.

   \begin{figure} [htpb]   
\begin{center}
   	  \includegraphics[width=10cm]{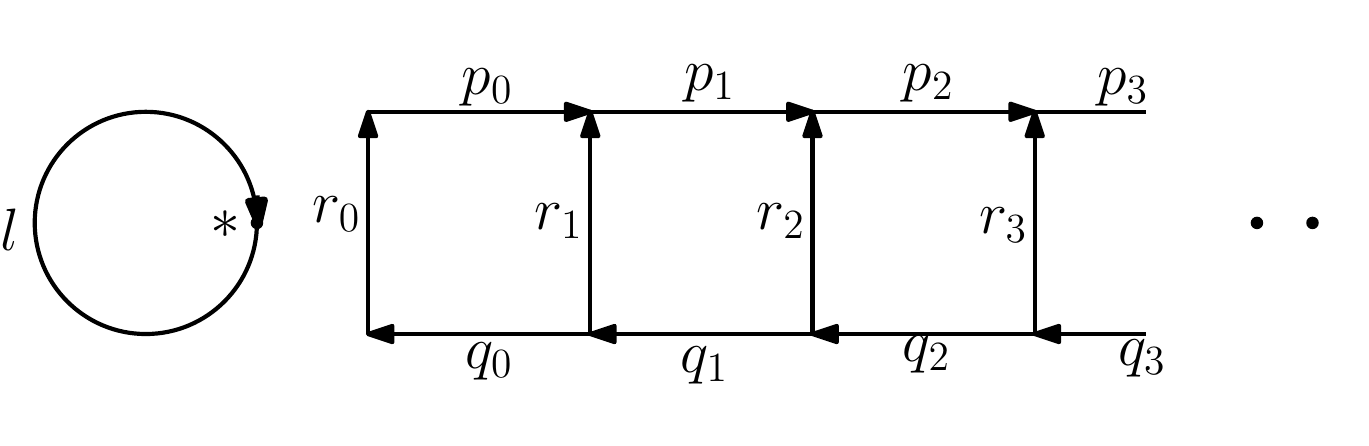}
   	  \caption{The graph $G$}
   	  \label{fig:labelled_ladder}
\end{center}
   \end{figure}

As usual, we denote the vertex set of $G$ by $V$ and the edge set by $E$. We call the unique vertex lying on the loop at the left $*$.

\begin{thm}
$M^+$ is a thin sums matroid over the field $\Qbb$.
\end{thm}
\begin{proof}
We begin by specifying the family $(f_e | e \in E)$ of functions from $V$ to $\Qbb$ for which $M^+ = M_f$. We take $f_e$ to be $f^G_e$ as in \autoref{rep_A} if $e$ is one of the $p_i$ or $q_i$, to be $\chi_*$ if $e = l$, and to be $f^G_e + i\cdot \chi_*$ if $e = r_i$.

First, we have to show that every circuit of $M^+$ is dependent in $M_f$. There are a variety of possible circuit types: in fact, types $(b)$, $(c)$, $(e)$ and $(f)$ from Figures \ref{fig:fin_cyc} and \ref{fig:alg_cyc} can arise. We shall only consider type $(f)$: the proofs for the other types are very similar. \autoref{fig:special_circuit} shows the two ways a circuit of type $(f)$ can arise.

   \begin{figure} [htpb]   
\begin{center}
   	  \includegraphics[width=10cm]{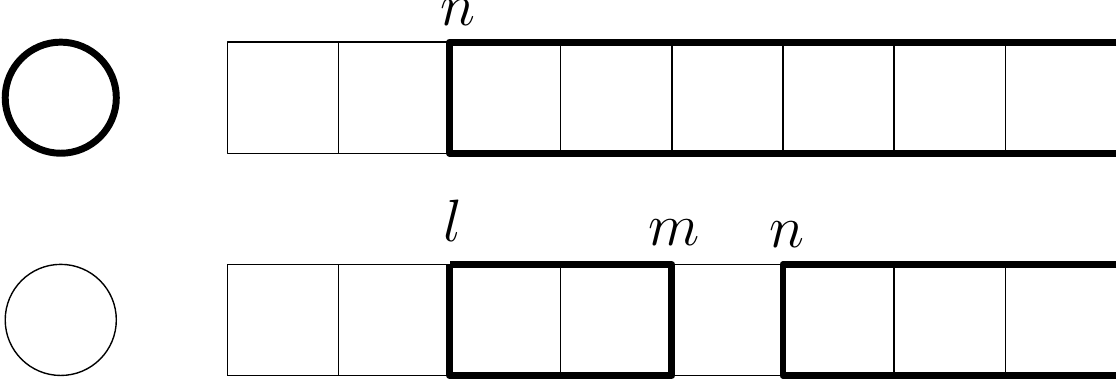}
   	  \caption{The two ways of obtaining a circuit of type $(f)$}
   	  \label{fig:special_circuit}
\end{center}
   \end{figure}

The first includes the edge $l$, together with $r_n$ for some $n$ and all those $p_i$ and $q_i$ with $i \geq n$. We seek a thin dependence $\lambda$ such that $\lambda$ is nonzero on precisely these edges.

We shall take $\lambda_{r_n} = 1$. We can satisfy the equations $\sum_{e \in E}
\lambda_ef_e(v)$ with $v \neq *$ 
by taking $\lambda_{p_i} = \lambda_{q_i} = 1$ for all $i \geq n$. The equation
$\sum_{e \in E} \lambda_e f_e(*) = 0$ reduces to $\lambda_* + n\lambda_{r_n} =
0$, which we can satisfy by taking $\lambda_* = -n$. It is immediate
that this gives a thin dependence of $f$. 

The second way a circuit of type $(f)$ can arise includes the edges $r_l$, $r_m$ and $r_n$, together with those $p_i$ and $q_i$ with either $l \leq i < m$ or $n \leq i$. We seek a thin dependence $\lambda$ such that $\lambda$ is nonzero on precisely these edges. 

The equations $\sum_{e \in E} \lambda_ef_e(v)$ with $v \neq *$ may be satisfied by taking $\lambda_{p_i} = \lambda_{q_i} = \lambda_{r_l} = -\lambda_{r_m}$ for $l \leq i \leq m$ and $\lambda_{p_i} = \lambda_{q_i} = \lambda_{r_n}$ for $i \geq n$. The equation $\sum_{e \in E} \lambda_e f_e(*) = 0$ reduces to $l\lambda_{r_l} + m \lambda_{r_m} + n \lambda_{r_n} = 0$, which since $\lambda_{r_m} = - \lambda_{r_l}$ reduces further to $(m - l)\lambda_{r_m} = n\lambda_{r_n}$. We can satisfy this equation by taking $\lambda_{r_m} = n$ and $\lambda_{r_n} = m - l$. 
Taking the remaining $\lambda_e$ to be given as above then gives a thin dependence of $f$.
Note that $\lambda\neq 0$ since $m\neq l$ and thus $\lambda_{r_n}\neq0$. 

Next, we need to show that every dependent set of $M_f$ is also dependent in $M^+$, completing the proof. Let $D$ be such a dependent set, as witnessed by a nonzero thin dependence $\lambda$ of $f$ which is 0 outside $D$. Let $D' = \{e | \lambda_e \neq 0\}$, the {\em support} of $D$. Using the equations $\sum_{e \in E} \lambda_ef_e(v)$ with $v \neq *$, we may deduce that the degree of $D'$ at each vertex (except possibly $*$) is either 0 or at least 2. Therefore any edge (except possibly $l$) contained in $D'$ is contained in some circuit of $M$ included in $D'$. Since $\{l\}$ is already a circuit of $M$, we can even drop the qualification `except possibly $l$'.

Since $D'$ is nonempty, it must include some circuit $O$ of $M$. Suppose first of all for a contradiction that $D' = O$. The intersection of $D'$ with the set $\{l\} \cup \{r_i | i \in \Nbb_0\}$ is nonempty, so by the equation $\sum_{e \in E}\lambda_e f_e(*) = 0$ this intersection must have at least 2 elements. The only way this can happen with $D'$ a circuit is if there are $m < n$ such that $D'$ consists of $r_m$, $r_n$, and the $p_i$ and $q_i$ with $m \leq i < n$. We now deduce, since $\lambda$ is a thin dependence, that $\lambda_{p_i} = \lambda_{q_i} = \lambda_{r_m} = -\lambda_{r_n}$ for $m \leq i \leq n$. In particular, the equation $\sum_{e \in E} \lambda_e f_e(*) = 0$ reduces to $(m - n) \lambda_{r_m} = 0$, which is the desired contradiction as by assumption $\lambda_{r_m} \neq 0$ and $m<n$. Thus $D' \neq O$, and we can pick some $e \in D \setminus O$. As above, $D'$ includes some  $M$-circuit $O'$ containing $e$. Then the union $O \cup O' \subseteq D$ is $M^+$-dependent by Corollary \ref{2circs}. 
\end{proof}

\begin{thm}
$(M^+)^*$ is not a thin sums matroid over any field.
\end{thm}
\begin{proof}
Suppose for a contradiction that it is a thin sums matroid $M_f$, with $f$ a family of functions $A \to k$. For each circuit $O$ of $(M^+)^*$, we can find a nonzero thin dependence $\lambda$ of $f$ which is nonzero only on $O$ - it must be nonzero on the whole of $O$ by minimality of $O$.

The circuits of $(M^+)^* = (M^*)^-$ are precisely the circuits and the bases of $M^*$, the dual of the algebraic cycle matroid of $G$, since no circuit in $M^*$ includes a base. This dual $M^*$, called the {\em skew cuts} matroid of $G$, is known to have as its circuits those cuts of $G$ which are minimal subject to the condition that one side contains no rays.

Thus since $\{r_0, p_0\}$ is a skew cut, we can find a thin dependence $\lambda^0$ which is nonzero precisely at $r_0$ and $q_0$. Similarly, for each $i > 0$ we can find a thin dependence $\lambda^i$ which is nonzero precisely at $q_{i-1}$, $r_i$ and $q_i$. Since the set of bold edges in \autoref{fig:second_wild} is also a circuit of $(M^+)^*$, there is a thin dependence $\lambda$ which is nonzero on precisely those edges.

To obtain a contradiction, we will show that $\{r_i | i \in \Nbb\}$ is dependent in $M_f$.
The idea behind the following calculations is to consider $\{r_i | i \in \Nbb\}$ as the limit 
of the $M_f$-circuits $\{r_i | 0\leq i\leq n\}\cup \{p_n\}$ and then to use the properties of thin sum representations
to show that the ``limit'' $\{r_i | i \in \Nbb\}$ inherits the dependence.

Now define the sequences $(\mu_i| i \in \Nbb)$ and $(\nu_i | i \in \Nbb)$ inductively by $\nu_0 = 1$, $\nu_i = -(\lambda^i_{p_i}/\lambda^i_{p_{i-1}})\nu_{i-1}$ for $i > 0$ and $\mu_i = -(\lambda^i_{r_i} / \lambda^i_{p_i}) \nu_i$. Pick any $a \in A$. Then we have $0 = \sum_{e \in E}\lambda^0_ef_e(a) = \lambda^0_{r_0} f_{r_0}(a) + \lambda^0_{p_0} f_{p_0}(a)$, and rearranging gives
$$\nu_0f_{p_0}(a) = \mu_0f_{r_0}(a) \, .$$

Similarly, $0 = \sum_{e \in E}\lambda^i_ef_e(a) = \lambda^i_{p_{i-1}} f_{p_{i-1}}(a) +  \lambda^i_{r_i} f_{r_i}(a) + \lambda^i_{p_i} f_{p_i}(a)$, and rearranging gives
$$\nu_if_{p_i}(a) = \nu_{i-1}f_{p_{i-1}}(a)  + \mu_if_{r_i}(a) \, .$$
So by induction on $i$ we get the formula
$$\nu_if_{p_i}(a) = \sum_{j = 0}^i \mu_jf_{r_j}(a) \, .$$

The formula $\sum_{e \in E} \lambda_e f_e(a) = 0$ implicitly includes the statement that the sum is well defined, 
so only finitely many summands can be nonzero. In particular, there can only be finitely many $i$ for which $f_{p_i}(a) \neq 0$. 
It then follows by the formula above that there are only finitely many $i$ such that $f_{r_i}(a)$ is nonzero, 
since if $f_{r_i}\neq 0$, then as $\mu_i\neq0$ we have $\nu_i f_{p_i}(a)\neq \nu_{i-1}f_{p_{i-1}}(a)$.
So as $\nu_i\neq 0$ and $\nu_{i-1}\neq 0$, one of $f_{p_i}(a)$ or $f_{p_{i-1}}(a)$ is not equal to zero.
Therefore all but finitely many $f_{r_i}(a)$ are zero since all but finitely
many $f_{p_i}(a)$ are zero.
So the following sum is well defined and evaluates to zero.

$$\sum_{i = 0}^{\infty} \mu_i f_{r_i}(a) = 0 \, .$$

Therefore, if we define a family $(\lambda'_e | e \in E)$ by $\lambda'_{r_i} = \mu_i$ and $\lambda'_e = 0$ for other values of $e$, then we have
$$\sum_{e \in E}\lambda'_e f_e(a) = 0 \, .$$

Since $a \in A$ was arbitrary, this implies that $\lambda'$ is a thin dependence of $f$. 
Note that $\lambda'\neq0$ since $\lambda'_{r_0}\neq0$.
Thus the set $\{r_i | i \in \Nbb\}$ is dependent in $M_f = (M^*)^-$. But it is
also an $(M^*)^-$basis, since adding $l$ 
gives a basis of $M^*$. This is the desired contradiction.
\end{proof}

\bibliographystyle{plain}
\bibliography{literatur}

\end{document}

%% file: minimalJ.tex
\usepackage{epsfig}
\usepackage{graphicx}
\usepackage{amsbsy}
\usepackage{amsmath}
\usepackage{amsfonts}
\usepackage{amssymb}
\usepackage{textcomp}
\usepackage{hyperref}
\usepackage{aliascnt}

\newcommand{\mcm}[3]{\newcommand{#1}[#2]{{\ensuremath{#3}}}} 

\mcm{\tuple}{1}{\langle #1 \rangle}
\mcm{\name}{1}{\ulcorner #1 \urcorner}
\mcm{\Nbb}{0}{\mathbb{N}}
\mcm{\Zbb}{0}{\mathbb{Z}}
\mcm{\Rbb}{0}{\mathbb{R}}
\mcm{\Cbb}{0}{\mathbb{C}}
\mcm{\Qbb}{0}{\mathbb{Q}}
\mcm{\Bcal}{0}{\cal B}
\mcm{\Ccal}{0}{\cal C}
\mcm{\Dcal}{0}{\cal D}
\mcm{\Ecal}{0}{\cal E}
\mcm{\Fcal}{0}{\cal F}
\mcm{\Gcal}{0}{\cal G}
\mcm{\Hcal}{0}{\cal H}
\mcm{\Ical}{0}{\cal I}
\mcm{\Lcal}{0}{\cal L}
\mcm{\Mcal}{0}{\cal M}
\mcm{\Ncal}{0}{\cal N}
\mcm{\Pcal}{0}{{\cal P}}
\mcm{\Scal}{0}{{\cal S}}
\mcm{\Tcal}{0}{{\cal T}}
\mcm{\Ucal}{0}{{\cal U}}
\mcm{\Vcal}{0}{{\cal V}}
\mcm{\Ycal}{0}{{\cal Y}}
\mcm{\Mfrak}{0}{\mathfrak M}

\mcm{\restric}{0}{\upharpoonright}
\mcm{\upset}{0}{\uparrow}
\mcm{\onto}{0}{\twoheadrightarrow}
\mcm{\smallNbb}{0}{{\small \mathbb{N}}}
\DeclareMathOperator{\preop}{op}
\mcm{\op}{0}{^{\preop}}

{\begin{array}{c}
\setlength{\unitlength}{1em}}%
{\end{array}}

\usepackage{amsthm}

\newcommand{\theoremize}[2]{\newaliascnt{#1}{thm} \newtheorem{#1}[#1]{#2} \aliascntresetthe{#1}}

\theoremstyle{plain}
\newtheorem{thm}{Theorem}[section]
\theoremize{lem}{Lemma}
\theoremize{sublem}{Sublemma}
\theoremize{claim}{Claim}
\theoremize{rem}{Remark}
\theoremize{prop}{Proposition}
\theoremize{cor}{Corollary}
\theoremize{que}{Question}
\theoremize{oque}{Open Question}
\theoremize{con}{Conjecture}

\theoremstyle{definition}
\theoremize{dfn}{Definition}
\theoremize{eg}{Example}
\theoremize{exercise}{Exercise}
\theoremstyle{plain}